\documentclass[10pt]{article}
\usepackage[matrix,arrow,curve]{xy}
\usepackage{amssymb,amsthm}
\usepackage{supertabular}

\usepackage{graphicx}
\usepackage{pstricks}
\usepackage{psfrag}

\usepackage{hyperref}

\usepackage{epic}
\usepackage{eepic}
\newtheorem{dummy}{anything}[section]
\newtheorem{Theorem}[dummy]{Theorem}
\newtheorem{Lemma}[dummy]{Lemma}
\newtheorem{Proposition}[dummy]{Proposition}

\newtheorem{Example}[dummy]{Example}

\newtheorem{Remark}[dummy]{Remark}
\newtheorem{Remarks}[dummy]{Remarks}

\newtheorem{ccote}[dummy]{}

\newcommand{\bbr}{{\mathbb R}}

\newcommand{\bbz}{{\mathbb Z}}

\newcommand{\cala}{{\mathcal A}}
\newcommand{\calb}{{\mathcal B}}

\newcommand{\calg}{{\mathcal G}}
\newcommand{\calh}{{\mathcal H}}

\newcommand{\caln}{{\mathcal N}}

\newcommand{\calp}{{\mathcal P}}

\newcommand{\calt}{{\mathcal T}}

\newcommand{\mancqfd}{\hfill \ensuremath{\Box}}
\newcommand{\llangle}[2]{\langle #1 ,#2\rangle}
\newcommand{\pcirc}{\kern .7pt {\scriptstyle \circ} \kern 1pt}

\newcommand{\mun}{{-1}}

\newcommand{\fpp}{\ensuremath{\hookrightarrow}}

\newcommand{\beq}[1]{\begin{equation}\label{#1}}
\newcommand{\eeq}{\end{equation}}

\newcommand{\scr}{\scriptscriptstyle}

\renewcommand{\:}{\colon}

\newcommand{\sk}[1]{\vskip #1 mm}
\newcommand{\eqref}[1]{(\ref{#1})}
\newcommand{\hfl}[2]{\smash{\mathop{\hbox to 1 truecm{\kern %
3pt\rightarrowfill\kern 3pt}}%
\limits^{\scriptstyle#1}_{\scriptstyle#2}}}
\newcommand{\cqfd}{\unskip\kern 6pt\penalty 500%
\raise -2pt\hbox{\vrule\vbox to10pt{\hrule width %
4pt\vfill\hrule}\vrule}\smallskip}
\newcommand{\proref}[1]{Proposition~\ref{#1}}
\newcommand{\remref}[1]{Remark~\ref{#1}}
\newcommand{\lemref}[1]{Lemma~\ref{#1}}

\newcommand{\thref}[1]{Theorem~\ref{#1}}
\newcommand{\exref}[1]{Example~\ref{#1}}
\newcommand{\secref}[1]{Section~\ref{#1}}

\newcommand{\dfn}[1]{{\it #1}}

\newcommand{\dia}[1]{\begin{array}{c}{\xymatrix@C-3pt@M+2pt@R-4pt{#1 }}\end{array}}

\newcommand{\bool}{Boolean}
\newcommand{\sym}[1]{{\rm Sym}_{\textstyle #1}}
\newcommand{\bbun}{{\mathbf 1}}
\newcommand{\comeq}[1]{\hbox{{\footnotesize #1}}}
\newcommand{\peri}[1]{\lfloor #1 \rceil}

\newcommand{\nua}[2]{\caln^{#1}_{#2}}
\newcommand{\sho}{{\rm Sh}}

\title{Counting polygon spaces, \\ Boolean functions and majority games}
\author{Jean-Claude HAUSMANN}
\date{}

\begin{document}
\maketitle 

\begin{abstract}
We explain why numbers occurring in the classification of polygon spaces coincide with numbers
of self-dual equivalence classes of threshold functions, or regular \bool\ functions, or
of decisive weighted majority games.
\end{abstract}

\section{Introduction}
Initiated by K.~Walker \cite{Wa},
the classification of polygon spaces with $n$ edges (see \cite{HR} and \secref{S.polyg} hereafter) 
involves chambers delimited by a hyperplane arrangement in $(\bbr_{>0})^n$ and 
so-called virtual genetic codes. The number $c(n)$ of chambers
modulo coordinate permutations and the number $v(n)$ of virtual genetic codes were computed by several authors
(see \cite{HRweb}) and the currently known figures are as follows

\sk{4}

\hskip -0mm
\begin{minipage}{110mm}
\begin{tabular}{c|cccccccccccc} \small
$n$ &  \footnotesize 3&\footnotesize 4&\footnotesize 5&\
\footnotesize 6 &\footnotesize 7&\footnotesize 8&\footnotesize 9 &\footnotesize 10 &\footnotesize 11
\\[1mm]\hline \rule{0mm}{4mm}
\small $c(n)$ &\footnotesize 2 &\footnotesize 3&
\footnotesize 7&\footnotesize \kern 2.9pt 21&\footnotesize 135&
\footnotesize 2,470&  \footnotesize  175,428 &\footnotesize 52,980,624 &\footnotesize ?
\\[1mm]\hline \rule{0mm}{4mm}
\small $v(n)$  &\footnotesize 2 &\footnotesize 3&
\footnotesize 7&\footnotesize \kern 2.9pt 21&\footnotesize 135&
\footnotesize 2,470&  \footnotesize  319,124 &  \footnotesize 1,214,554,343 &
\footnotesize $\sim 1.7\cdot 10^{15}$
\end{tabular}
\end{minipage}
\sk{3}\noindent
(more precisely: $v(11)=1,\!706,\!241,214,185,942$, computed by Minfeng Wang: see \cite{HRweb}).
According to the {\em On-Line Encyclopedia of Integer Sequences (OEIS)}, these numbers occur in other sequences:

\begin{ccote}\label{OLEIS2} \rm 
 The numbers $c(n)$ of chambers up to permutation coincide with the 
{\em Numbers of self-dual equivalence classes of threshold functions of $n$ or fewer variables}, or the 
{\em numbers of majority (i.e., decisive and weighted) games with $n$ players}, listed in \cite{OLEISsdth}.
\end{ccote}

\begin{ccote}\label{OLEIS1} \rm
The numbers $v(n)$ of virtual genetic codes coincide with the 
{\em numbers of Boolean functions of $n$ variables that are self-dual and regular},
listed for $n\leq 10$ in \cite{OLEISsdr}.
\end{ccote}

The aim of this note is to explain these numerical coincidences by constructing natural bijections between
the sets under consideration. In particular, the above mentioned precise value $v(11)$
may be added in \cite{OLEISsdth}. The principal results are Propositions~\ref{Pa}, \ref{P.games} and~\ref{P2}

The paper is organized as follows. \secref{S.caltn} presents the transformation group used in various equivalence
relations. \secref{S.polyg} recalls the notations and the classification's result for polygon spaces.
In Sections~\ref{S.sptri} and~\ref{S.games}, we introduce threshold functions and majority games and prove
the bijections involved in~\ref{OLEIS2}, while \secref{S.sdf} concerns the case of~\ref{OLEIS1}.
Finally, we treat in \secref{S.ngen} the case of non-generic polygon spaces, giving rise to an apparently unknown
integer sequence.

\sk{1}
I thank Matthias Franz for drawing my attention to this problem and for useful conversations.

\section{The transformation group $\calt_n$}\label{S.caltn}

In this section, we define the transformation group $\calt_n$, responsible for several equivalence relations
occurring in this paper. Incidentally, a few notation are introduced, which are used throughout the next sections

Fix a positive integer $n$. 
If $X$ is a set, the symmetric group $\sym{n}$ acts on $X^n$ by permuting the components.
This is a right action: 
an element $x\in X^n$ is formally a map $x:\{1,\dots,n\}\to X$ ($x_i=x(i)$) 
and $\sigma\in\sym{n}$ acts by pre-composition, i.e.
$x^\sigma=x\pcirc\sigma$. Note that right actions are most often denoted exponentially in this paper. 

Let $\cala_n=(\bbz_2)^n$, the elementary abelian group of rank $n$ denoted additively. 
The $\sym{n}$-action on $\cala_n$ gives rise to the semi-direct product
\begin{equation}\label{EdefTn}
\calt_n = \cala_n \rtimes \sym{n} \, .
\end{equation}
Recall that, as a set, $\calt_n$ coincides with $\cala_n\times\sym{n}$. 
We we may use the short notations $\nu=(\nu,{\rm id})$
and $\sigma=(0,\sigma)$ (which enables us to consider $\sym{n}$ as a subgroup of $\calt_n$). 
The group $\calt_n$ is thus generated by $\nu\in\cala_n$ and $\sigma\in\sym{n}$, subject
to the relations $\sigma^\mun\nu\sigma=\nu^\sigma$. Note the formulae 
$(\nu,\sigma)(\mu,\tau)=(\nu\mu^{\sigma^\mun},\sigma\tau)$ and $(\nu,\sigma)^\mun=(\nu^\sigma,\sigma^\mun)$.

The group $\calt_n$ will act on several sets. We finish this section with a few examples.

\begin{Example}\label{EactRn} The action of $\calt_n$ on $\bbr^n$ \rm is 
defined as follows: if $z=(z_1,\dots,z_n)\in\bbr^n$,
the $i$-th component of $z^{(\nu,\sigma)}$ is   
\begin{equation}\label{EDacrRn}
(z^{(\nu,\sigma)})_i = (-1)^{\nu_{\sigma(i)}} z_{\sigma(i)} \, .
\end{equation}
In particular, $z^\nu=\big((-1)^{\nu_1}z_1,\dots,(-1)^{\nu_n}z_n\big)$ and $z^\sigma=(z_{\sigma(1)},\dots,z_{\sigma(n)})$.
The following lemma will be useful.
\end{Example}

\begin{Lemma}\label{Ld1}
The inclusion $(\bbr_{\geq 0})^n\hookrightarrow  \bbr^n$ induces a bijection on the orbit sets
$$
(\bbr_{\geq 0})^n\big/ \sym{n} \stackrel{\approx}{\longrightarrow} \bbr^n\big/\calt_n  \, .
$$
\end{Lemma}

\begin{proof}
Suppose that $z'=z^{(\nu,\sigma)}$. 
If $z_i\geq 0$ and $z_i'\geq 0$, Formula~\eqref{EDacrRn} implies that $\nu_i=1$ only if $z'_i=0=(-1)^{\nu_i}z_{\sigma(i)}$,
in which case $\nu_i$ may be replaced by $0$ without changing $z'$. 
Hence, $z'=z^{(\nu,\sigma)}=z^\sigma$, which implies that our map is injective. 
By Formula~\eqref{EDacrRn} again, each $\calt_n$-orbit contains an element $a$ with $a_i\geq 0$, so the map is also surjective. 
\end{proof}

\begin{Example}\label{EactAn} The action of $\calt_n$ on \bool\ vectors. \ \rm
We consider another copy of $\bbz_2^n$ called $\calb_n$, the set of $n$-tuples $(x_1,\dots,x_n)$ of \bool\ variables.
The set $\calb_1$ is thus $\{true,false\}$, with its usual numerisation $true=1$, $false=0$, ${\rm xor=+}$, {\it etc}.
We sometimes use binary strings, e.g. $1010$ for $(1,0,1,0)$. 
The addition law of $\bbz_2^n$ produces a right action $\calb_n\times\cala_n\to\calb_n$ of $\cala_n$ on $\calb_n$.
Note that the action of $1$ on $x_i$ is $(x_i+1)_{{\rm mod}\, 2} = \bar x_i$, the \dfn{negation} of $x_i$ ($\bar 0=1$ and $\bar 1=0$).
This is the reason for which an element of $\cala_n$ is, in this paper, denoted by $\nu=(\nu_1,\dots,\nu_n)$, 
the letter $\nu$ standing for {\it negation}. 
Another useful equality is $\bar x_i=1-x_i$ (viewing $\{0,1\}\subset\bbr$). 
This action extends to an action of $\calt_n$ on $\calb_n$ by the formula   
$$
\big(x^{(\nu,\sigma)}\big)_i = x_{\sigma(i)}+ \nu_{\sigma(i)} \, . 
$$ 
\end{Example}

\begin{Example}\label{EactPn} 
The action of $\calt_n$ on $\calp(\underline{n})$, \rm where $\calp(\underline{n})$ is the set of
subsets of~$\underline{n}$. We use the bijection $\chi\:\calp(\underline{n})\to\calb_n$ 
associating to $J\subset\underline{n}$ its \dfn{characteristic} $n$-tuple $\chi(J)$, whose $i$-th component is
$$
\chi(J)_i = \chi(J)(i) = {\rm truth}(i\in J)
$$
(i.e. $\chi(J)(i)=1$ if and only if $i\in J$). Note that $\chi(J)+\chi(K)=\chi(J{\scriptstyle\bigtriangleup} K)$,
where $\scriptstyle{\bigtriangleup}$ denotes the symmetric difference.
The $\calt_n$-action on $\calp(\underline{n})$ is defined so that 
$\chi$ is equivariant, using the $\calt_n$-action of \exref{EactAn}: $\chi(J^{(\nu,\sigma)}) = \chi(J)^{(\nu,\sigma)}$.
This amounts to the formulae $J^\nu=J{\scriptstyle\bigtriangleup}\chi^\mun(\nu)$, $J^\sigma=\sigma^\mun(J)$ and thus
$$
J^{(\nu,\sigma)} = \sigma^\mun(J{\scriptstyle\bigtriangleup}\chi^\mun(\nu)) \, .
$$
\end{Example}

\section{Polygon spaces}\label{S.polyg}
\setcounter{equation}{0}

In this section, we recall the notations for polygon spaces and their classification (see \cite{HR} or \cite[\S~10.3]{Ha}).
Fix two integers $n$ and $d$ and set $\underline{n}=\{1,2\dots,n\}$.
For $a=(a_1,\dots,a_n)\in\bbr^n$, the \dfn{polygon space} $\nua{n}{d}(a)$ is defined by
\begin{equation}\label{E.defpolsp}
\nua{n}{d}(a) = \Big\{z\in (S^{d-1})^n\,\big|\,\llangle{a}{z}=0\Big\} \bigg/SO(d) \, ,
\end{equation}
where $\llangle{\kern 1pt}{}$ denotes the standard scalar product in $\bbr^n$.
Classically, this definition is restricted to $a\in(\bbr_{>0})^n$, in which case an 
element of $\nua{n}{d}(a)$ may be visualized 
as a configuration of $n$ successive segments in $\bbr^d$, of length
$a_1,\dots,a_{n}$, starting and ending at the origin.
The vector $a$ is thus called the \dfn{length vector}. 
Following some recent works (see e.g. \cite{Fr}), we take advantage of 
Definition~\eqref{E.defpolsp} making sense for $a\in\bbr^n$.
In most of the cases, this extension does not create 
new polygon spaces up to homeomorphism (see \remref{R.newSp}).

The classification of polygon spaces up to homeomorphism is based on the stratification induced
by the \dfn{tie hyperplane arrangement} (or just \dfn{tie arrangement}) $\calh(\bbr^n)$ in  $\bbr^n$
$$
\calh(\bbr^n)=\{\calh_J\mid J\subset \underline{n}\}  \, ,
$$
where the \dfn{$J$-tie hyperplane} $\calh_J$ is defined by
$$
\calh_J:=\Big\{(a_1,\dots,a_n)\in\bbr^n \Bigm|
\sum_{i\in J}a_i=\sum_{i\notin J}a_i\Big\}.
$$
A tie hyperplane is often called just a \dfn{wall}.
The stratification associated to $\calh=\calh(\bbr^n)$ is defined by the filtration
$$\{0\} = \calh^{(0)}\subset
\calh^{(1)}\subset\cdots\subset\calh^{(n)}=\bbr^n,
$$
with $\calh^{(k)}$ being
the subset of those $a\in\bbr^n$ which belong to at least
$n-k$ distinct walls $\calh_J$. A {\it stratum} of dimension
$k$ is a connected component of $\calh^{(k)}-\calh^{(k-1)}$.
Note that a stratum of dimension $k\geq 1$ is an open convex cone
in a $k$-plane of $\bbr^n$.
Strata of dimension $n$ are called {\it chambers} and their elements are called \dfn{generic}.

Note that the tie arrangement $\calh(\bbr^n)$ is invariant under the action of $\calt_n$.
Indeed, using the tools of \exref{EactPn}, one checks
that $(\calh_J)^{(\nu,\sigma)}= \calh_K$ with $K=\sigma\big(J{\scr\bigtriangleup}\chi^\mun(\nu)\big)$.

We may restrict the stratification $\calh$ to $V=(\bbr_{>0})^n$, $(\bbr_{\neq 0})^n$, $(\bbr_{\geq0})^n$ or $\bbr^n$.
Each of these choices for V gives rise to a set
$\mathbf{Ch}(V)$ of chambers, contained in the set $\mathbf{Str}(V)$ of corresponding strata.

\begin{Proposition}\label{PstratCH}
The inclusions  $(\bbr_{>0})^n\hookrightarrow (\bbr_{\geq0})^n \hookrightarrow\bbr^n$ descend to bijections
$$
\mathbf{Ch}((\bbr_{>0})^n)/\sym{n} \stackrel{\approx}{\longrightarrow} 
\mathbf{Ch}((\bbr_{\geq0})^n)/\sym{n}
\stackrel{\approx}{\longrightarrow}  \mathbf{Ch}(\bbr^n)/\calt_n 
$$
and
$$
\mathbf{Str}((\bbr_{\geq0})^n)/\sym{n}
\stackrel{\approx}{\longrightarrow}  \mathbf{Str}(\bbr^n)/\calt_n \, .
$$
\end{Proposition}
 
\begin{proof}
These bijections are direct consequences of \lemref{Ld1}, except for the first one
$\mathbf{Ch}((\bbr_{>0})^n)/\sym{n}\to\mathbf{Ch}((\bbr_{\geq0})^n)/\sym{n}$. The latter is obviously
injective. For the surjectivity, we use that if $\varepsilon>0$ is small enough and $a$ is generic, we can
replace the zero components of $a$ by $\varepsilon$ without leaving the chamber of $a$ (see e.g. \cite[\S\,2.1]{HauGDPS}). 
\end{proof}

\begin{Remarks}\label{R.newstr}\rm 
(a) The map $\mathbf{Str}((\bbr_{>0})^n)/\sym{n}\to \mathbf{Str}((\bbr_{\geq0})^n)/\sym{n}$ is clearly injective
but it is not surjective. The stratum $\{(0,\dots,0)\}$ is of course not in the image but also other less degenerate
strata, such as the intersection of the two walls $\calh_{\{2\}}\cap\calh_{\{3\}}$ in $\mathbf{Str}((\bbr_{\geq0})^3)$.
Indeed, the equations $a_2=a_1+a_3$ and $a_3=a_1+a_2$ imply that $a_1=0$. See also~\ref{comptStr}

(b) In \cite{HR} and in the lists of \cite{HRweb}, \dfn{conventional representatives} for classes in 
$\mathbf{Ch}((\bbr_{>0})^n)/\sym{n}$ are used, allowing zero components in length-vectors.
These zeros stand there for any small enough positive numbers (e.g. $(0,1,1,1)$ means $(\varepsilon,1,1,1)$). 
This actually uses the first bijection 
of \proref{PstratCH} (see also its proof). Thus, in our new setting, these conventional representatives are bona fide 
representative of $\mathbf{Ch}((\bbr_{\geq0})^n)/\sym{n}$. 
However, using zero length is unsuitable for the symplectic geometry of spatial polygon spaces (see e.g. \cite{HK}).
\end{Remarks}

The main theorem for the classification of polygon spaces \cite[Theorem~1.1]{HR} generalizes, 
with the same proof, in the following statement.

\begin{Theorem}\label{T.polcla}
Let $a,a'\in\bbr^n$. If $a$ and $a'$ are two representatives of the same class in  $\mathbf{Str}(\bbr^n)/\calt_n$
then $\nua{n}{d}(a)$ and $\nua{n}{d}(a')$ are homeomorphic. \mancqfd
\end{Theorem}

\begin{Remark}\rm 
For generic $a$ and certain $n$ and $d$, the converse of \thref{T.polcla} is true: if 
$\nua{n}{d}(a)$ and $\nua{n}{d}(a')$ are homeomorphic, then 
$a$ and $a'$ represent the same class in  $\mathbf{Ch}(\bbr^n)/\calt_n$
(see, e.g. \cite{FHS,Sc,Sc2}).
\end{Remark}

\begin{Remark}\label{R.newSp}\rm 
By \proref{PstratCH} and \thref{T.polcla}, taking generic length vectors in $\bbr^n$ 
instead of $(\bbr_{>0})^n$ does not produce new polygon spaces up to homeomorphism. 
This is the same for non-generic length vectors, provided that they are taken in $(\bbr_{\neq 0})^n$ 
(if $a\in(\bbr_{\neq 0})^n$,  then $b=a^\nu\in(\bbr_{>0})^n$ for some $\nu\in\cala_n$
and the map $x\mapsto x^\nu$ gives a homeomorphism from $\nua{n}{d}(b)$ to $\nua{n}{d}(a)$).
But the non-generic strata of 
$\calh(\bbr_{\geq 0})^n)$ (see \remref{R.newstr}.(a)) produce, in general, new polygon spaces.
For example, for $(0,1,1)\in\calh_{\{2\}}\cap\calh_{\{3\}}\in\mathbf{Str}((\bbr_{\geq0})^3)$, one has
$$
\begin{array}{rcll}
\nua{3}{3}(0,1,1) &=& \{(x,y,z)\in(S^2)^3\mid y+z=0\}/SO(3) \\[2mm] &\approx & 
(S^2\times S^2)/SO(3) \\[2mm] &\approx &
 pt \times S^2/SO(2) \approx [-1,1] \, , 
\end{array}
$$
which is not homeomorphic to a polygon space $\nua{3}{3}(a)$ for $a\in (\bbr_{>0})^3$.
Indeed,  $\mathbf{Str}((\bbr_{>0})^3)/\sym{3}$ contains $3$ strata, giving
$\nua{3}{3}(0,0,1)=\emptyset$, $\nua{3}{3}(1,1,2)=pt$ and $\nua{3}{3}(1,1,1)=pt$. 
\end{Remark}

We now restrict ourselves to to the generic case. If $a\in\bbr^n$ is generic then
$\sum_{i\in J}a_i\neq\sum_{i\notin J}a_i$ for all $J\subset\underline{n}$. 
When $\sum_{i\in J}a_i<\sum_{i\notin J}a_i$, 
the set $J$ is called \dfn{$a$-short} (or just \dfn{short}) and 
its complement is \dfn{$a$-long} (or just \dfn{long}).
Short subsets form a subset $\sho(a)$ of $\calp(\underline{n})$. 
Define $\sho_n(a)=\{J\in\{1,\dots,n-1\}\mid J\cup\{n\} \in \sho(a)\}$.
As $J$ is short if and only if $\bar J=\underline{n}-J$ is long,
the set $\sho_n(a)$ determines $\sho(a)$. 
Indeed, either $n\in J$ or $n\in\bar J$, and thus  
$\sho_n(a)$ tells us whether $J\in\sho(a)$ or $\bar J\in\sho(a)$.

The chamber of $a$ is obviously determined by $\sho(a)$ and thus by $\sho_n(a)$.  
This permits us to characterize  $\alpha\in\mathbf{Ch}((\bbr_{\geq0})^n)/\calt_n$ by
a subset ${\rm gc}(\alpha)$ of $\calp(\underline{n})$,
called the \dfn{genetic code} of $\alpha$. 
For this, we consider, using \proref{PstratCH}, the only representative $a$ of $\alpha$ such that
$0\geq a_1 \geq\cdots\geq  a_n$.
Define a partial order ``\fpp'' on
$\calp(\underline{m})$ by saying that $A\fpp B$ if and only if there exits
a non-decreasing map $\varphi:A\to B$ such that $\varphi(x)\geq x$.
Note that, if $B$ is short, so is $A$ if $A\fpp B$. 

The {\it genetic code} ${\rm gc}(\alpha)$ of $\alpha$ is the set of
elements $A_1,\dots,A_k$ of $\sho_n(a)$ which are
maximal with respect to the order ``\fpp''. The chamber of $a$ (and thus $\alpha$)
is determined by  ${\rm gc}(\alpha)$. 
We also use the notation ${\rm gc}(a)$ for ${\rm gc}(\alpha)$.
For instance ${\rm gc}((0,0,1))=\{\emptyset\}$, ${\rm gc}((0,1,1,1))= \{\{4,1\}\}$,
${\rm gc}((1,1,2,3,3,5))=\{\{6,2,1\},\{6,3\}\}$, etc (see \cite{HR}).

An algorithm is designed in \cite{HR} to list all the possible genetic code. 
A first step is 
to observe that all $A,B\in{\rm gc}(\alpha)$ satisfy
\begin{enumerate}\renewcommand{\labelenumi}{(\alph{enumi})}
\item $A\not\fpp B$ if $A\neq B$ and
\item $\bar A\not\fpp A$.
\end{enumerate}
Indeed, Condition (a) holds true by maximality of $A$ and, 
if $\bar A\fpp A$, then $A$ would be both short and long which is impossible. 
A finite set $\{A_1,\dots,A_k\}\subset \calp_m(\underline{n})$,
satisfying Conditions (a) and (b) is called a {\it virtual genetic code} (of type~$n$).

\section{Self-dual threshold functions}\label{S.sptri}
\setcounter{equation}{0}

Fix a positive integer $n$.
We use the set $\calb_n$ of \bool\ vectors with its $\calt_n$-action introduced in \exref{EactAn}.
A \bool\ function on $n$ variables is a map $\calb_n\to \calb_1=\bbz_2$. 
The group $\calt_n$ acts on the right on the set of 
\bool\ functions by $f^{(\nu,\sigma)}(x)= f(x^{(\nu,\sigma)^\mun})$, which gives the formula
\begin{equation}\label{E.aonfns}
f^{(\nu,\sigma)}(x_1,\dots,x_n) = (x_{\sigma^\mun(1)}+\nu_1,\dots,x_{\sigma^\mun(n)}+\nu_n) \, .
\end{equation}
This $\calt_n$-action on the set of \bool\ functions produces the equivalence relation used in \ref{OLEIS1}.

A \bool\ function $f$ is called \dfn{self-dual} if $f(\bar x)=\overline{f(x)}$.  

\begin{Lemma}\label{La}
Self-duality is preserved by the action of $\calt_n$.
\end{Lemma}

\begin{proof}
Let $f\:\calb_n\to\calb_1$ be a \bool\ function which is self-dual and let  $(\nu,\sigma)\in\calt_n$.
Note that $\bar x = x^{(\bbun,{\rm id})}$ where $\bbun=(1,\dots,1)$. Then,
$$
\begin{array}{rcll}
f^{(\nu,\sigma)}(\bar x) &=& f^{(\nu,\sigma)}(x^{(\bbun,{\rm id})}) \\[2mm] &=&
f(x^{(\bbun,{\rm id})(\nu,\sigma)^\mun}) \\[2mm] &=&
f(x^{(\nu,\sigma)^\mun(\bbun,{\rm id})}) & \comeq{since $(\bbun,{\rm id})$ is in the center of $\calt_n$}\\[2mm] &=&
f\big(\overline{x^{(\nu,\sigma)^\mun}}\big) \\[2mm] &=&
\overline{f(x^{(\nu,\sigma)^\mun})} & \comeq{since $f$ is self-dual} \\[2mm]  &=&
\overline{f^{(\nu,\sigma)}(x)} \, ,
\end{array}
$$
which proves that $f^{(\nu,\sigma)}$ is self-dual.
\end{proof}

Let $w=(w_1,\dots,w_n)\in\bbr^n$ and $t\in \bbr$. We consider the \bool\ function
$$
f_{(w,t)}(x) = {\rm truth}\big(\llangle{x}{w} \geq t \big) =
\left\{
\begin{array}{lll}
1 & \hbox{if } \llangle{x}{w} \geq t \\
0 & \hbox{otherwise,}
\end{array}\right.
$$
where $\llangle{\kern 1pt}{}$ denotes the standard scalar product in $\bbr^n$. The function $f_{(w,t)}$
is called the \dfn{threshold function} with \dfn{weights} $w_1,w_2,\dots,w_n$ and \dfn{threshold} $t$ 
(see \cite{Hn} and \remref{rkA} below). Reference \cite{Hn} emphasize the importance of threshold functions
in neuron-like systems. The few subsequent lemmas gather several properties of threshold functions. 

\begin{Lemma}\label{Lb}
Threshold functions are preserved by the action of $\calt_n$. More precisely,
let $(w,t)\in\bbr^n\times\bbr$ and let $(\nu,\sigma)\in\calt_n$. 
Then $f_{(w,t)}^{(\nu,\sigma)}=f_{(w',t')}$ where
$$
w'=w^{(\nu,\sigma)}  \ \hbox{ and }  \ t'=t - \llangle{\nu}{w^{\sigma}} \, .  
$$
\end{Lemma}

\begin{proof}
As we are dealing with an action of $\calt_n$, it is enough to prove the lemma for $\sigma=(0,\sigma)$ and
for $\nu=(\nu,{\rm id})$.
In the first case, one has 
$$
f_{(w,t)}^\sigma(x) = {\rm truth}\big(\llangle{x^{\sigma^\mun}}{w} \geq t \big)
= {\rm truth}\big(\llangle{x}{w^\sigma} \geq t \big) = f_{(w^\sigma,t)} \, .
$$
In the other case, one first prove, using the truth tables of $x_i$ and $\nu_i$, that
$$
(x_i)^{\nu_i} w_i = x_i (-1)^{\nu_i}w_i + \nu_iw_i \, .
$$
This implies that
$\llangle{x^\nu}{w}=\llangle{x}{w^\nu}+\llangle{\nu}{w}$. 
Therefore,
$$
\begin{array}{rcll}
f_{(w,t)}^\nu(x)  &=&  f_{(w,t)}(x^\nu)  \\[2mm]  &=&
{\rm truth}\big(\llangle{x^{\nu}}{w}\geq t \big)  \\[2mm]  &=&
{\rm truth}\big(\llangle{x}{w^\nu}\geq t - \llangle{\nu}{w}\big)  \\[2mm]  &=&
f_{(w',t')}(x) \, .  
\end{array}  
$$
\end{proof}

For $w\in\bbr^n$, define $\peri{w}= \frac{1}{2}\sum_{i=1}^n w_i$. 
The relationship between threshold and self-dual functions is the following.

\begin{Lemma}\label{Lsd}
A threshold function  $f_{(w,t)}$ is self-dual if and only if the following two conditions hold.
\begin{itemize}
\item[(a)]   $f_{(w,t)}=f_{(w,\peri{w})}$  and 
\item[(b)]  $\llangle{x}{w}\neq \peri{w}$ for any $x\in\calb_n$.
\end{itemize}
\end{Lemma}

\begin{proof}
Conditions (a) and (b) are clearly sufficient for $f_{(w,t)}$ being self-dual. Conversely, if $f=f_{(w,t)}$ is self-dual, then
$$
\llangle{x}{w} \geq t \ \Leftrightarrow \ 
f(x)=1  \ \Leftrightarrow \ f(\bar x)=0 \ \Leftrightarrow \  \llangle{\bar x}{w} < t \, .
$$
This, together with the same argument exchanging $x$ and $\bar x$, proves that 
$$
\llangle{x}{w} < \llangle{\bar x}{w} \ \hbox{ or } \ \llangle{x}{w} > \llangle{\bar x}{w}
$$
for all $x\in\calb_n$. As $\llangle{x}{w} + \llangle{\bar x}{w} = 2\peri{w}$, this proves (a) and (b). 
\end{proof}

\begin{Lemma}\label{Lperi}
Let $w\in\bbr^n$ and $(\nu,\sigma)\in\calt_n$. Then $f_{(w,\peri{w})}^{(\nu,\sigma)} = f_{(w',\peri{w'})}$
(where $w'$ is given by \lemref{Lb}).
 \end{Lemma}

\begin{proof}
As we are dealing with an action of $\calt_n$, it is enough to prove the lemma for $(0,\sigma)$ and $(\nu,{\rm id})$. 
The case $(0,\sigma)$ is straightforward by \lemref{Lb} and the equality $\peri{w^{\sigma^\mun}}=\peri{w}$. 
For $(\nu,{\rm id})$, \lemref{Lb} again tells us that $f_{(w,\peri{w})}^{(\nu,{\rm id})} = f_{(w',t')}$ with
$w'=w^{\nu}$ and 
$$
t' = \peri{w} -\llangle{\nu}{w} = \frac{1}{2}\big(\sum w^i -2\llangle{\nu}{w}\big)  = \peri{w^\nu} \, . \qedhere
$$
\end{proof}

We are ready to prove the main result of this section.
To a generic $a\in\bbr^n$, we associate the threshold function $f_{(a,\peri{a})}$, 
which is self dual by \lemref{Lsd}. If $a$ and $a'$ belong to the same chamber of $\calh(\bbr^n)$,
then
$$
f_{(a,\peri{a})}^\mun(\{0\}) = \sho(a) = \sho(a') = f_{(a',\peri{a'})}^\mun(\{0\}) \, ,
$$
which proves that $f_{(a,\peri{a})}=f_{(a',\peri{a'})}$. We thus get a map
$$
\tilde\Xi:\mathbf{Ch}(\bbr^n) \to \mathbf{SDT}(n) \, ,
$$
where $\mathbf{SDT}(n)$ denotes the set of self-dual threshold functions on $\calb_n$.

\begin{Proposition}\label{Pa}
The above map $\tilde\Xi$ descends to a bijection
$$
\Xi:\mathbf{Ch}(\bbr^n)/\calt_n  \stackrel{\approx}{\longrightarrow} \mathbf{SDT}(n)/\calt_n  \, .
$$
\end{Proposition}

\begin{proof}
By \lemref{Lperi}, the map $\tilde\Xi$ is $\calt_n$-equivariant, so the orbit map $\Xi$ is well defined.
It is surjective by \lemref{Lsd}. 
For the injectivity, let $\alpha$ and $\beta$ be two chambers, represented by length vectors $a$ and $b$. 
If $\Xi(\alpha)=\Xi(\beta)$, then $f_{(b,\peri{b})}=f_{(a,\peri{a})}^{(\nu,\sigma)}$ 
for some $(\nu,\sigma)\in\calt_n$.  By \lemref{Lperi}, one has 
$f_{(b,\peri{b})}=f_{(a^{(\nu,\sigma)},\peri{a^{(\nu,\sigma)}})}$. Therefore
$$
\sho\big(a^{(\nu,\sigma)}\big) = 
f_{(a^{(\nu,\sigma)},\peri{a^{(\nu,\sigma)}})}^\mun(\{0\}) = 
f_{(b,\peri{b})}^\mun(\{0\}) = \sho(b) \, ,
$$
which implies that $\beta=\alpha^{(\nu,\sigma)}$.
\end{proof}

\begin{Remarks}\label{rkA}\rm
 (a) There are several variants in the literature for the definition of a threshold function, for instance
by requiring that $w_i$ and/or $t$ be integers (see e.g. \cite[p.~75]{Kn}). 
As chambers contain integral representative, \proref{Pa} holds true as well for these versions. 

(b) Variables corresponding to zero weights are idle for $f_{(a,\peri{a})}$, so the latter
depends on fewer than $n$ variables. This is the reason of the words ``$n$ or fewer variables'' in \ref{OLEIS2}.

(c) Composing the bijection $\Xi$ of \proref{Pa} with the bijection\\ 
$\mathbf{Ch}((\bbr_{>0})^n)/\sym{n} \stackrel{\approx}{\longrightarrow} \mathbf{Ch}(\bbr^n)/\calt_n$ of \proref{PstratCH},
produces a bijection
\begin{equation}\label{E.bij1}
\mathbf{Ch}((\bbr_{>0})^n)/\sym{n} \stackrel{\approx}{\longrightarrow}  \mathbf{SDT}(n)/\calt_n  \, .
\end{equation}
This explains why the numbers of the first line of the table in the introduction are equal to those of~\cite{OLEISsdth}.
\end{Remarks}

By \proref{PstratCH}, the bijection of \eqref{E.bij1} factors through the bijection\\ 
$\mathbf{Ch}((\bbr_{>0})^n)/\sym{n} \stackrel{\approx}{\longrightarrow}  \mathbf{SDT}(n)/\calt_n$. 
A direct consequence is the following lemma, which will be useful later.

\begin{Lemma}\label{Ld}
Let $(w,t)$ and $(w',t')$ be two elements of $\bbr^n\times\bbr$. Suppose that $w$ and $w'$ are generic and that 
$f_{(w,t)}$ and $f_{(w',t')}$ are in the same $\calt_n$-orbit.
If $w_i\geq 0$ and $w_i'\geq 0$ for all $i\in\underline{n}$, then $w$ and $w'$ are in the same $\sym{n}$-orbit.
\end{Lemma}

\section{Decisive weighted majority games}\label{S.games}
\setcounter{equation}{0}

In this section, we describe the equivalence between $\mathbf{Ch}((\bbr_{\geq 0})^n)/\sym{n}$ and 
the strategic equivalence classes of 
decisive weighted majority games with $n$ players, as mentioned in \ref{OLEIS2}. Our references for game theory are
\cite[Chapter~10]{vNM} and \cite{Is}.

A \dfn{game} on a set of $n$ players (indexed by $\underline n$) is
a set of subsets of $\underline{n}$ called the \dfn{winning sets}, such that any set containing a winning
set is also winning. The game is \dfn{decisive} (or \dfn{simple}) if 
for any $S\subset \underline{n}$, either $S$ or its complement is winning but not both.
Two games are \dfn{strategically equivalent} if there is a bijection
between their players identifying the families of winning sets.
A game with $n$ players may be extended to $m>n$ players by adding $m-n$ ``voteless'' players or \dfn{dummies}:
a subset $S$ of $\underline{m}$ is thus wining if and only if $S\cap\underline{n}$ is wining. 

We see that a game $\calg$ defines and is determined by a \bool\ function $f_{\calg}\:\calb_n\to\calb_1$,
given by $f(x)=1$ if and only if $x$ is the characteristic $n$-tuple of a set $S\subset\underline{n}$
which is a winning set. 
As any superset of a winning wins, the function $f_\calg$ is \dfn{monotone}, i.e. its value does not change from~$1$ 
to~$0$ when any of its variables changes from~$0$ to~$1$ \cite[p.~55]{Kn}.
That $\calg$ is decisive translates into $f_\calg$ being self-dual. 
Two games $\calg$ and $\calg'$ are  
strategically equivalent if and only if $f_{\calg}$ and  $f_{\calg'}$ are in the same $\sym{n}$-orbit.

A game $\calg$ is a \dfn{weighted majority game} if there exists $w\in\bbr^n$ such that $f_{\calg}=f_{(w,\peri{w})}$. 
We write $\calg=\calg(w)$. If $w_i\leq 0$, then the player $i$ is a dummy.
Indeed, since every superset of a winning
set wins,  $w_i\leq 0$ implies that $|w_i|$ is so small that it makes no difference. 
Therefore, one can replace the negative weights by $0$ without changing the strategic equivalence class of the game.
We can thus suppose that no weight is negative. 
By \lemref{Ld}, two decisive majority games $\calg(w)$ and $\calg(w')$ are then strategically equivalent if and only if 
$w$ and $w'$ are in the same $\sym{n}$-orbit. 

Note that if $a\in(\bbr_{\geq0})^n$ is a length vector, the winning set of $\calg_a$ are the $a$-long subsets of $\underline{n}$. 
Also, $a$ is generic if and only if $\calg_a$ is decisive.
The above considerations, together with Propositions~\ref{PstratCH} and~\ref{Pa}, gives the following result.

\begin{Proposition}\label{P.games}
The map $a\mapsto\calg(a)$ induces a bijection from  $\mathbf{Ch}((\bbr_{\geq 0}^n))/\sym{n}$ to the set of strategic equivalence classes
of decisive weighted majority games.
\end{Proposition}

\section{Self-dual regular \bool\ functions}\label{S.sdf}
\setcounter{equation}{0}

A partial order on $\calb_n$ is defined by saying that $x\preceq y$ if 
$x_1+\cdots x_k \leq y_1+\cdots y_k$ for $1\leq k\leq n$ \cite[p.~92]{Kn}.
Note that  $x\preceq y$ if and only if $\bar y\preceq\bar x$: indeed, $x_1+\cdots+x_k \leq y_1+\cdots+y_k$
if and only if $k-x_1-\cdots -x_k \geq k-y_1-\cdots -y_k$.

A \bool\ $f\:\calb_n\to\calb_1$ is \dfn{regular} if $f(x)\preceq f(y)$ whenever $x\preceq y$ \cite[p.~93]{Kn}. 
For example, a threshold function $f_{(w,t)}$ is regular if $w_1\geq w_2 \geq\cdots\geq w_n\geq 0$. 
Let $\mathbf{SDR}(n)$ be the set of self dual regular \bool\ functions on $\calb_n$.

\begin{Proposition}\label{P2}
There is a bijection between $\mathbf{SDR}(n)$ and the set of virtual genetic codes (see \secref{S.polyg}). 
\end{Proposition}

Before proving \proref{P2}, we note the following lemma, in which $\calb_n^1=\{x\in\calb_n\mid x_1=1\}$.

\begin{Lemma}\label{Lsdbun}
A self-dual \bool\ $f\:\calb_n\to\calb_1$ is determined by its restriction to $\calb_n^1$.  
\end{Lemma}

\begin{proof}
We use that $x\in\calb_n^1$ if and only if $\bar x\notin\calb_n^1$. As $f$ is self-dual, one has 
\begin{equation}\label{Lsdbun-eq}
f(x) =
\left\{ 
\begin{array}{lll}
f(x) & \hbox{if $x\in\calb_n^1$} \\[1mm]
\overline{f(\bar x)} & \hbox{otherwise.} 
\end{array}\right.
\end{equation}
\end{proof}

\begin{proof}[Proof of \proref{P2}]
To $f\in\mathbf{SDR}(n)$, we associate its \dfn{code} $\gamma(f)$ which is a subset of $\calb_n^1$.
By definition, $\gamma(f)$ is the set of $\preceq$-maximal elements $x\in\calb_n^1$ for which $f(x)=0$. 
For instance $\gamma(f_{((2,1,1,1),5/2)})=\{1000\}$ and  $\gamma(f_{((2,2,2,1),7/2)})=\{1001\}$. 
If $\gamma(f)=\{b_1,\dots,b_k\}$, then
\begin{itemize}
\item[(i)] $b_i\not\preceq b_j$ for all $i\neq j$ and
\item[(ii)] $\bar b_i\not\preceq b_j$ for all $i,j$.
\end{itemize}
Let $\Gamma_n$ be the set of subsets of $\calb_n^1$ satisfying (i) and (ii). 
We first establish that $\gamma\:\mathbf{SDR}(n)\to\Gamma_n$ is a bijection. 
Indeed, $\gamma(f)$ clearly determines the restriction of $f$ to $\calb_n^1$, and then determines $f$ by \lemref{Lsdbun};
this proves that $\gamma$ is injective. For the surjectivity, let $R\in\Gamma_n$. The formula
$$
f(x) = 
\left\{ 
\begin{array}{lll}
0 & \hbox{if $\exists\ r\in R$ with $x\preceq r$} \\[1mm]
1 & \hbox{otherwise} 
\end{array}\right.
$$
defines a function on $\calb_n^1$ which can be extended to $f\:\calb_n\to\calb_1$ by \eqref{Lsdbun-eq}.
Such a definition guarantees that $f$ is self-dual. For the regularity, let $x\preceq y$ be two elements in $\calb_n$.
The condition $f(x)\preceq f(y)$ is automatic if $f(y)=1$. We can thus assume that $f(y)=0$, so we must prove that $f(x)=0$.
There are four cases.

{\it Case 1: $x_1=1=y_1$.} \  As $f(y)=0$, there exists $r\in R$ with $y\preceq r$. As $x\preceq y$, then $x\preceq r$ and thus $f(x)=0$.

{\it Case 2: $x_1<y_1$.} \ If $f(x)=1$, then $\bar x \preceq r$ for some $r\in R$. 
As $f(y)=0$, there exists $s\in R$ with $y\preceq s$. Therefore, $\bar r \preceq x \preceq y \preceq s$ which contradicts (ii). 

{\it Case 3: $x_1> y_1$}. This is impossible since $x\preceq y$. 

{\it Case 4: $x_1=0=y_1$}. If $f(x)=1$, then $f(\bar x)=0$. As $f(\bar y)=1$ and $\bar y\preceq \bar x$, the pair $(\bar y,\bar x)$
would contradict Case~1, already established.

\sk{1}

It remains to establish a bijection from $\Gamma_n$ and the set of virtual genetic codes. 
To $x\in\calb_n^1$ one associates $x^\sharp\subset \underline{n}$ by the rule
$$
x^\sharp = \{i\in \underline{n} \mid x_{n+1-i} = 1 \} \, .
$$
For instance, $(1000)^\sharp = \{4\}$ while $(1010)^\sharp = \{2,4\}$. Obviously, $x\mapsto x^\sharp$ 
is a bijection between $\calb_n^1$ and the subsets of $\underline{n}$ containing $n$. Conditions (i) and (ii) above 
are intertwined with Conditions (a) and (b) of \cite[p.~37]{HR}. The latter define a virtual genetic code. 
Hence, the correspondence $x\mapsto x^\sharp$ maps $\Gamma_n$ bijectively to the set of virtual genetic codes.
\end{proof}

\begin{Remark}\rm
For $n\geq 9$, not every self-dual regular \bool\ function is  equivalent to a threshold function. As an example,
the function with $\gamma(f)=\{100101010\}$, corresponding to the genetic code $\{9,6,4,2\}$ (see \cite[Lemma 4.5]{HR}).
\end{Remark}

\section{Non-generic strata}\label{S.ngen}
\setcounter{equation}{0}

In this section, we give an analogue of the bijection $\Xi$ of \proref{Pa}, extended to possibly non-generic strata, 
taking advantage of 3-valued (3V) \bool\ functions.
A \dfn{3V-\bool\ function} is a map $f\:\calb_n\to \{-1,0,1\}$. It is \dfn{self-dual} if $f(\bar x) = -f(x)$.
As in \secref{S.sptri}, a right $\calt_n$-action on the set of 3V-\bool\ functions using \eqref{E.aonfns}.
As in \lemref{La}, one proves that self-duality is preserved by this $\calt_n$-action.

Let $w=(w_1,\dots,w_n)\in\bbr^n$ and $t\in \bbr$. The 3V-\bool\ function
$$
f_{(w,t)}^{3V}(x) = 
\left\{
\begin{array}{rll}
1 & \hbox{if } \llangle{x}{w} > t \\
0 & \hbox{if } \llangle{x}{w} = t \\
-1 & \hbox{if } \llangle{x}{w} < t 
\end{array}\right.
$$
is called the \dfn{3V-threshold function} with \dfn{weights} $w_1,w_2,\dots,w_n$ and \dfn{threshold} $t$. 
With essentially the same proofs, Lemma~\ref{La}--\ref{Ld} remain valid without change for 3V-threshold functions, 
except for \lemref{Lsd} which requires the hypothesis $0\notin {Image}(f_{(w,t)}^{3V})$. 
If this is not the case, one easily proves the following lemma.

\begin{Lemma}\label{lsdii}
 Let $(w,t)\in\bbr^n\times\bbr$ such that $0\in {Image}(f_{(w,t)}^{3V})$. Then, $f_{(w,t)}^{3V}$ is self-dual
if and only if $t=\peri{w}$. \mancqfd
\end{Lemma}

Let $\mathbf{SDT}^{3V}(n)$ be the set of self-dual 3V-threshold functions on $\calb_n$. 
As in \secref{S.sptri}, we define a map $\tilde\Xi^{3V}\:\mathbf{Str}(\bbr^n)\to \mathbf{SDT}^{3V}(n)$
by associating associating to $S\in \mathbf{Str}(\bbr^n)$ the 3V-threshold function
$f_{(a,\peri{a})}^{3V}$ for $a\in S$. We check that $\tilde\Xi^{3V}$ is well defined and $\calt_n$-equivariant, 
thus inducing a map $\tilde\Xi^{3V}\:\mathbf{Str}(\bbr^n)/\calt_n\to \mathbf{SDT}^{3V}(n)/\calt_n$.
The same proof as for \proref{Pa} gives following proposition.

\begin{Proposition}\label{Pb}
The map $\Xi^{3V}:\mathbf{Str}(\bbr^n)/\calt_n  \to \mathbf{SDT}^{3V}(n)/\calt_n$ is a bijection.
\mancqfd
\end{Proposition}

\begin{Remark}\rm The relationship between the maps $\Xi$ and $\Xi^{3V}$ of Propositions~\ref{Pa} and \ref{Pb}
is as follows.
There is an obvious injection $j\:\mathbf{SDT}(n)/\calt_n\to \mathbf{SDT}^{3V}(n/\calt_n)$
induced by $f\mapsto  \epsilon\pcirc f$ where $\epsilon(u)=(-1)^u$. Its image is the set of 3V-\bool\ functions $f$ such that 
$0\notin {\rm Image}(f)$. One has a commutative diagram
$$
\begin{array}{c}{\xymatrix@C-3pt@M+2pt@R-4pt{%
\mathbf{Ch}(\bbr^n)/\calt_n\kern 2pt  \ar@{>->}[r]   \ar[d]_{\Xi}^{\approx}  &  
\mathbf{Str}(\bbr^n)//\calt_n   \ar[d]^{\Xi^{3V}}_{\approx} \\
\mathbf{SDT}(n)/\calt_n\kern 2pt \ar@{>->}[r]^(.47){j} & \mathbf{SDT}^{3V}(n)/\calt_n
}}\end{array} .
$$
%
\end{Remark}

\begin{ccote}\label{comptStr} Computing the number of strata. \rm \ 
Consider the following numbers
\begin{itemize}
\item  $c(n) = \sharp\big(\mathbf{Ch}(\bbr^n)/\calt_n\big)$  
\item  $k(n) 
=\sharp\big(\mathbf{Str}((\bbr_{\neq 0})^n)/\calt_n\big)$
\item  $tk(n) = \sharp\big(\mathbf{Str}(\bbr^n)/\calt_n  \big)$.
\end{itemize}
 
For example, $c(1)=0$ and $k(1)=tk(1)=1$ (the stratum of $(0)$). For $n=2$, one has 
$c(2)=1$ (the chamber of $(0,1)$), $k(2)=2$ (the previous chamber and the stratum of $(1,1)$), while
$tk(2)=3$ because the stratum of $(0)$ in  $\mathbf{Str}(\bbr^1)$ gives rise to that of $(0,0)$ 
in $\mathbf{Str}(\bbr^2)$. In general, the injection 
$\bbr^{n-1}\approx\{0\}\times\bbr^{n-1}\hookrightarrow\bbr^n$ induces an injection
$[\mathbf{Str}(\bbr^{n-1})-\mathbf{Ch}(\bbr^{n-1})]/\calt_{n-1} 
\hookrightarrow \mathbf{Str}(\bbr^{n})/\calt_{n}$. This proves the recursion formula
\begin{equation}\label{E.rec}
tk(n) = 
k(n) + tk(n-1)- c(n-1)  \, .
\end{equation}
The number $k(n)$ was computed in \cite[\S\,5]{HR} for $n\leq 8$. Thanks to \proref{PstratCH},
the values of $c(n)$ may be taken from the table in the introduction. Using Formula~\eqref{E.rec},
we thus get the following table.

\begin{center}
\begin{minipage}{110mm}
\begin{tabular}{c|cccccccccccc} \small
$n$ &  \footnotesize 1 & \footnotesize 2 &\footnotesize 3&\footnotesize 4&\footnotesize 5&\
\footnotesize 6 &\footnotesize 7&\footnotesize 8
\\[1mm]\hline \rule{0mm}{4mm}
\small $c(n)$ & \footnotesize 0 & \footnotesize 1 &\footnotesize 2 &\footnotesize 3&
\footnotesize 7&\footnotesize \kern 2.9pt 21&\footnotesize 135&
\footnotesize 2,470 
\\[1mm]\hline \rule{0mm}{4mm}
\small $k(n)$ &\footnotesize 1 & \footnotesize 2
&\footnotesize 3 &\footnotesize 7&
\footnotesize 21&\footnotesize \kern 2.9pt 117&\footnotesize 1506&
\footnotesize 62254&  
\\[1mm]\hline \rule{0mm}{4mm}
\small $tk(n)$  & \footnotesize 1 & \footnotesize 3
&\footnotesize 5 &\footnotesize 10 &
\footnotesize 28 &\footnotesize 138 &\footnotesize 1623 &
\footnotesize 63742
\end{tabular}
\end{minipage}
\end{center}
\end{ccote}

The sequences $k(n)$ and $tk(n)$ do not seem to occur in the {\em On-Line Encyclopedia of Integer Sequences}.



\sk{4}\noindent {\small
Jean-Claude HAUSMANN\\
Math\'ematiques -- Universit\'e\\
B.P. 64, 
CH--1211 Geneva 4, Switzerland\\
jean-claude.hausmann@unige.ch}

\end{document}